\documentclass[12pt]{article}
\usepackage{authblk}
\usepackage{amsmath,amssymb,amsfonts,latexsym,amsthm,enumerate,color,cite}
\usepackage[colorlinks=true, citecolor=blue, linkcolor=blue, urlcolor=blue]{hyperref}

\newtheorem{formula}{}[section]

\newtheorem{corollary}[formula]{Corollary}
\newtheorem{lemma}[formula]{Lemma}
\newtheorem{theorem}[formula]{Theorem}

\title{Structure of $k$-closures of finite nilpotent permutation groups}
\author{Dmitry Churikov}
\affil{Sobolev Institule of Mathematics, Novosibirsk, Russia}
\affil{Novosibirsk State University, Novosibirsk, Russia}

\begin{document}
\maketitle

\begin{abstract}
Let $G$ be a permutation group on a set $\Omega$, and $k$ a positive integer. The $k$-closure $G^{(k)}$ of $G$ is the largest subgroup of $\operatorname{Sym}(\Omega)$, with the same as $G$ orbits of componentwise action on $\Omega^k$. We prove that the $k$-closure of a finite nilpotent permutation group is the direct product of $k$-closures of its Sylow subgroups.
\end{abstract}

\section{Itroduction}

Let $G$ be a permutation group on a finite nonempty set $\Omega$, and $k$ a positive integer. Denote by $\operatorname{Orb}_k(G)$ the set of the orbits of componentwise action of $G$ on the Cartesian power $\Omega^k$. Elements of the set $\operatorname{Orb}_k(G)$ are called $k$-orbits of $G$, and they form a partition of the set~$\Omega^k$. The $k$-closure $G^{(k)}$ of~$G$ is the automorphism group of $k$-orbits of $G$~\cite[Def. 5.3]{WieInvRel},
$$
G^{(k)}=\{g\in\operatorname{Sym}(\Omega)\mid O^g=O~\forall O\in\operatorname{Orb}_k(G)\}.
$$
Equivalently, $G^{(k)}$ is the largest subgroup of the symmetric group $\operatorname{Sym}(\Omega)$ such that $\operatorname{Orb}_k(G)=\operatorname{Orb}_k(G^{(k)})$.
The following inclusions for closures of $G$ are known~\cite[Thm.~5.8]{WieInvRel}:
\begin{equation}\label{ClosureSeries}
G\leq G^{(k)}\leq G^{(k-1)}, k\geq 2.
\end{equation}
The main result of this paper establishes a simple link between the $k$-closure of a finite nilpotent permutation group and $k$-closures of its Sylow subgroups.

\begin{theorem}\label{main1}
If $G$ is a finite nilpotent permutation group, and $k\geq 2$, then
$$
G^{(k)}=\prod_{P\in\operatorname{Syl}(G)} P^{(k)},
$$
i.e. the $k$-closure of $G$ is the direct product of $k$-closures of its Sylow subgroups. In particular, $G^{(k)}$ is nilpotent.
\end{theorem}

A group $G$ is $k$-closed if $G^{(k)}=G$. Theorem~\ref{main1} implies necessary and sufficient conditions for a nilpotent group to be $k$-closed.

\begin{corollary} A finite nilpotent permutation group is $k$-closed if and only if every Sylow subgroup of $G$ is $k$-closed.
\end{corollary}

Theorem~\ref{main1} generalizes recent results of~\cite{CP,CPr}, where similar theorems were proved for $2$-closures of nilpotent groups~\cite{CP} and for $k$-closures of abelian groups~\cite{CPr}. It is worth mentioning that nilpotency of $k$-closure of nilpotent group is already known and it easily follows from~\cite[Thm. 1.1]{CP} and the inclusions~(\ref{ClosureSeries}).

\section{Preliminaries}\label{sec:Section2}

We will mainly use standard notations of permutation group theory (see~\cite{Wielandt1964}). We especially note the following ones. If a group $G$ (an element $x\in G$) acts on a set $\Omega$, then $G^\Delta$ ($x^\Delta$ respectively) denotes the permutation group (the permutation respectively) induced by this action. If $G\leq\operatorname{Sym}(\Omega)$, and $\Delta\subseteq\Omega$ is a $G$-invariant subset, then $G_{\{\Delta\}} =\{x\in G\mid\Delta^x=\Delta\}$ is the setwise stabilizer of $\Delta$ in $G$. Finally, $1_\Delta$ denotes the identity element of $\operatorname{Sym}(\Delta)$.

Several times we will use the following criterion for a permutation to belong to a $k$-closure.

\begin{theorem}\label{W}{\rm \cite[Thm. 5.6]{WieInvRel}} If $G\leq \operatorname{Sym}(\Omega)$, $k\geq1$, and $x\in\operatorname{Sym}(\Omega)$, then $x\in G^{(k)}$ if and only if for all $\alpha_1, \dots, \alpha_k\in\Omega$ there exists $g\in G$ such that $\alpha_i^x = \alpha_i^g$, $i=1,\dots,k$.
\end{theorem}

Let $\pi$ be a set of prime numbers, and $n$ is a positive integer. Denote by $n_\pi$ the divisor of $n$ such that $(n_\pi,\frac{n}{n_\pi})=1$ and the set of prime divisors of $n_\pi$ is equal to $\pi$.

\begin{lemma}\label{CP1}{\rm \cite[Lemma 3.1]{CP}}
If $G$ is a finite nilpotent permutation group of degree $n$, and $H$ is a Hall subgroup of $G$, then
\begin{itemize}
    \item[{\em 1)}] the size of every orbit of $H$ is equal to $n_\pi$, where $\pi=\pi(H)$ is the set of prime divisors of $|H|$,
    \item[{\em 2)}] $G$ acts on $\operatorname{Orb}(H)$, and $H$ is the kernel of this action.
\end{itemize}
\end{lemma}

The following two lemmas are well known, but, apparently, they are not completely proved anywhere. We fill this gap. Some variations of these statements can be found in~\cite{3ClosSolv,EP,Elusive}.

\begin{lemma}\label{DirSum}
If $G_i\leq\operatorname{Sym}(\Omega_i)$, $i=1,2$, and the group $G_1\times G_2$ acts on the disjoint union $\Omega_1\cup\Omega_2$, then for all integers $k\geq 1$,
$$(G_1\times G_2)^{(k)} = G_1^{(k)}\times G_2^{(k)}.$$
\end{lemma}

\begin{proof}
$\boxed{\geq}$ It suffices to prove that $G_1^{(k)}\times 1$ and $1\times G_2^{(k)}$ are contained in $(G_1\times G_2)^{(k)}$. Fix a set of elements $\alpha_1,\ldots,\alpha_k\in\Omega_1\cup\Omega_2$ and its subset $\alpha_{j_1},\ldots,\alpha_{j_l}$ consisting of elements from $\Omega_1$. If $(g,1)\in G_1^{(k)}\times 1$, then the inclusions~(\ref{ClosureSeries}) imply that $g\in G_1^{(l)}$, and by Theorem~\ref{W} exists $h\in G_1$ such that $\alpha_{j_i}^g=\alpha_{j_i}^h$ for all $i=1,\ldots,l$.

Now consider the element $(h,1)\in G_1\times G_2$. By construction for all $i=1,\ldots,k$
\[
\alpha_i^{(h,1)}=\begin{cases}
\alpha_i^h=\alpha_i^g=\alpha_i^{(g,1)} \text{ if } \alpha_i\in\Omega_1,\\
\alpha_i=\alpha_i^{(g,1)} \text{ if } \alpha_i\in\Omega_2.
\end{cases}
\]
It follows from Theorem~\ref{W} that $(g,1)\in (G_1 \times G_2)^{(k)}$. So that, $G_1^{(k)}\times 1\leq (G_1\times G_2)^{(k)}$. The inclusion $1\times G_2^{(k)}\leq (G_1\times G_2)^{(k)}$ can be proved analogously.

$\boxed{\leq}$ Let $x\in (G_1 \times G_2)^{(k)}$, so $\Omega_i^x=\Omega_i$. For $i=1,2$ put $x_i=x^{\Omega_i}$. Then $x$ coincides with $(x_1,x_2)\in\operatorname{Sym}(\Omega_1\cup\Omega_2)$, which is acting on $\Omega_i$ as $x_i$, $i=1,2$. We show that $x_i\in G_i^{(k)}$ for $i=1,2$. Let $\alpha_1,\ldots,\alpha_k\in\Omega_i$. By Theorem~\ref{W} there exist $(h_1,h_2)\in G_1\times G_2$ such that $\alpha_j^x=\alpha_j^{(h_1,h_2)}$, $j=1,\ldots,k$. In particular, $\alpha_j^{x_i}=\alpha_j^{h_i}$ in view of $\alpha_j^{x_i}=\alpha_j^x=\alpha_j^{(h_1,h_2)}=\alpha_j^{h_i}$. Then Theorem~\ref{W} implies that $x_i\in G_i^{(k)}$, and $x\in G_1^{(k)}\times G_2^{(k)}$ then.
\end{proof}

\begin{lemma}\label{DirProd}
If $G_i\leq\operatorname{Sym}(\Omega_i)$, $i=1,2$, and the group $G_1\times G_2$ acts on the direct product $\Omega_1\times\Omega_2$, then for every integer $k\geq 1$, $$(G_1\times G_2)^{(k)} = G_1^{(k)}\times G_2^{(k)}.$$
\end{lemma}

\begin{proof}
$\boxed{\geq}$ It suffices to prove that $G_1^{(k)}\times 1$ and $1\times G_2^{(k)}$ are contained in $(G_1\times G_2)^{(k)}$. Let $(g,1)\in G_1^{(k)}\times G_2^{(k)}$, and $(\alpha_1,\beta_1),\ldots,(\alpha_k,\beta_k)\in\Omega_1\times\Omega_2$. Since $g\in G_1^{(k)}$, by Theorem~\ref{W} there exists $h\in G_1$ such that $\alpha_j^g=\alpha_j^h$, $j=1,\ldots,k$. This means that for the element $(h,1)\in G_1\times G_2$ the equality $(\alpha_j,\beta_j)^{(g,1)}=(\alpha_j^g,\beta_j)=(\alpha_j^h,\beta_j)=(\alpha_j,\beta_j)^{(h,1)}$ holds, and it follows from Theorem~\ref{W} that $(g,1)\in (G_1 \times G_2)^{(k)}$. So that, $G_1^{(k)}\times 1\leq (G_1\times G_2)^{(k)}$. Inclusion $1\times G_2^{(k)}\leq (G_1\times G_2)^{(k)}$ can be proved similarly.

$\boxed{\leq}$ First, we will study the structure of elements from $(G_1\times G_2)^{(2)}$. Put
$\Sigma_1=\{\{\alpha_1\}\times\Omega_2\mid\alpha_1\in\Omega_1\},$
$\Sigma_2=\{\Omega_1\times\{\alpha_2\}\mid\alpha_2\in\Omega_2\}.$
Obviously, the group $G_1 \times G_2$ acts on the sets $\Sigma_1$ and $\Sigma_2$. We will show that the group $(G_1 \times G_2)^{(2)}$ acts on these sets too, i.e. for all $x\in(G_1 \times G_2)^{(2)}$ and all $\Delta\in\Sigma_i$, $i=1,2$, either $\Delta^x=\Delta$, or $\Delta^x\cap \Delta=\varnothing$.

If $\Delta$ is a singleton, then there is nothing to proof. Suppose that there exist two different elements $u,v\in\Delta$ such that $u^x\in\Delta$, and $v^x\notin\Delta$. By Theorem~\ref{W}, there exists $h\in G_1 \times G_2$ such that $(u^x,v^x)=(u^h,v^h)$. But that would imply that $u^h\in\Delta$, and $v^h\notin\Delta$, which is impossible, because the group $G_1 \times G_2$ acts on the set $\Sigma_i$.

So the group $(G_1 \times G_2)^{(2)}$ acts on sets $\Sigma_1$ and $\Sigma_2$, and hence on the set $\Sigma_1\times\Sigma_2$. The bijection $$\rho:\Omega_1\times\Omega_2\rightarrow\Sigma_1\times\Sigma_2,\quad(\alpha_1,\alpha_2)\mapsto(\{\alpha_1\}\times\Omega_2,\Omega_1\times\{\alpha_2\}).$$
establishes a permutation isomorphism between the group $(G_1 \times G_2)^{(2)}$ and a subgroup of the group $\operatorname{Sym}(\Sigma_1\times\Sigma_2)$. For $i=1,2$ define the permutation $x_i$ acting on $\alpha\in\Omega_i$ in the following way:
\begin{center}
$\alpha^{x_1}=\beta\Leftrightarrow(\{\alpha\}\times\Omega_2)^x = \{\beta\}\times\Omega_2,$

$\alpha^{x_2}=\beta\Leftrightarrow(\Omega_1\times\{\alpha\})^x = \Omega_1\times\{\beta\}.$
\end{center}
Then the image of the permutation $x$ under the permutation isomorphism defined above coincides with the permutation $(x_1,x_2)\in\operatorname{Sym}(\Omega_1\times\Omega_2)$.

Let us return to the proof of the lemma. Let $x\in(G_1 \times G_2)^{(k)}$. From the previous observations and the inclusions~(\ref{ClosureSeries}) follows that $x=(x_1,x_2)$, $x_i\in\operatorname{Sym}(\Omega_i)$, $i=1,2$. We show that $x_i\in G_i^{(k)}$. Let $\alpha^i_1,\ldots,\alpha^i_k\in\Omega_i$. It follows from Theorem~\ref{W} that for $x\in (G_1 \times G_2)^{(k)}$ and for $(\alpha^1_1,\alpha^2_1),\ldots,(\alpha^1_k,\alpha^2_k)\in\Omega_1\times\Omega_2$ there exists $(h_1,h_2)\in G_1\times G_2$ such that $(\alpha^1_j,\alpha^2_j)^x=(\alpha^1_j,\alpha^2_j)^{(h_1,h_2)}$, $j=1,\ldots,k$. In particular, ${\alpha^i_j}^{x_i}={\alpha^i_j}^{h_i}$, $j=1,\ldots,k$, so Theorem~\ref {W} implies that $x_i\in G_i^{(k)}$, and $x\in G_1^{(k)}\times G_2^{(k)}$ then.
\end{proof}

The following technical lemma is of particular interest.

\begin{lemma}\label{SetwiseStabilizers}
Let $G$ be a finite nilpotent permutation group, $P\in\operatorname{Syl}(G)$, $\Delta_1,\ldots,\Delta_k\in \operatorname{Orb}(P)$, and $\Delta=\bigcup_{i=1}^k \Delta_i$. Then $$\left (\bigcap_{i=1}^k G_{\{\Delta_i\}}\right)^\Delta\leq P^\Delta.$$
\end{lemma}
\begin{proof} Let $g\in\left (\bigcap_{i=1}^k G_{\{\Delta_i\}}\right)^\Delta$. Since the group $G$ is nilpotent, $G=P\times H$, with $H$ being a Hall subgroup of $G$, and then $g=xy$ for some $x\in P$ and $y\in H$. The choice of $g$ and $x$ yields that $\Delta_i^g=\Delta_i^x=\Delta_i$ for all $i=1,\ldots,k$, so the same holds true for $y$, because $$\Delta_i=\Delta_i^g=\Delta_i^x=\Delta_i^{x^{-1}g}=\Delta_i^y.$$

We show that $y^{\Delta_i}=1_{\Delta_i}$ for all $i=1,\ldots,k$. Indeed, by the construction of $y$, the element $y^{\Delta_i}$ belongs to the centralizer $Z_i$ of the transitive group $P^{\Delta_i}\leq\operatorname{Sym}(\Delta_i)$, which is semiregular by~\cite[Exer.~4.5']{Wielandt1964}. Therefore $|Z_i|$ divides $|\Delta_i|$, which is a $p$-power. So that $Z_i$ is a~$p$-group. In particular, the order of the element $y^{\Delta_i} $ is $p$-power, and therefore $y^{\Delta_i}\in P^{\Delta_i}$. Since $P^{\Delta_i}\cap H^{\Delta_i}=1$, we have $y^{\Delta_i}=1_{\Delta_i}$.

Thus, $y^{\Delta}=1_{\Delta}$, which means that
$$
g^\Delta=(xy)^\Delta=x^\Delta y^\Delta=x^\Delta\in P^\Delta.
$$
\end{proof}

\newpage
\section{Proof of the theorem}\label{sec:Section3}

First, consider the case where the group $G$ is transitive. We use induction on $|\pi(G)|$. If $|\pi(G)|= 1$, then $G$ is a $p$-group, and by virtue of~\cite[Exer.~5.28]{WieInvRel} the $2$-closure $G^{(2)}$ is also a $p$-group, and due to the inclusions~(\ref{ClosureSeries}), $G^{(k)}$ is a $p$-group too for all~$k\geq2$.

Now let $G = P\times H$, where $P\in\operatorname{Syl}_p(G)$, and $H$ is the Hall subgroup of~$G$. Applying Lemma~\ref{CP1}(1) for orbits $\Delta\in\operatorname{Orb}(P)$ and $\Gamma\in\operatorname{Orb}(H)$, we obtain that
$$
|\Delta|=n_p \text{~and~} |\Gamma|=n_{p'},
$$
where $p'=\pi(H)$. Since $P$ and $H$ are normal subgroups of $G$, it follows that $ \Delta$ and $\Gamma$ are blocks of some imprimitivity system of the group~$G$. Moreover, the set $\Delta\cap\Gamma$ is either empty, or it is also a block whose order divides both~$|\Delta|$ and $|\Gamma|$, which means $|\Delta\cap\Gamma|\le 1$.\medskip

Now each point $\alpha\in\Omega$ can be uniquely associated with the orbits (containing this point) $\Delta_\alpha$ and $\Gamma_\alpha$ of the groups $P$ and $H$, respectively. Previous arguments imply that $|\Delta_\alpha\cap\Gamma_\alpha|=1$.

So the mapping
$$
\rho:\Omega\to \operatorname{Orb}(H)\times\operatorname{Orb}(\Gamma),\ \alpha\mapsto(\Delta_\alpha,\Gamma_\alpha)
$$
is a bijection. Denote by $P'$ and $H'$ the permutation groups induced by the action of $G$ on $\operatorname{Orb}(H)$ and $\operatorname{Orb}(P)$, respectively. By Lemma~\ref{CP1}(2) we obtain that
$$
P^\rho=P'\times 1\text{~and~} H^\rho=1\times H'.
$$
Thus, $G$ is permutationally isomorphic to the group $P'\times H'$ acting on the set $\operatorname{Orb}(H)\times\operatorname{Orb}(P)$. It follows from the Lemma~\ref{DirProd} that
$$
(P'\times H')^{(k)}=(P')^{(k)}\times (H')^{(k)}.
$$
The proof concludes by applying the induction assumption to the group $H$, which is possible by the equality $\pi(G)=\pi(P)\cup\pi(H)$.\medskip

Now let $G$ be intransitive. Each transitive constituent $H$ of $G$ is nilpotent (as a homomorphic image of a nilpotent group), and $\pi(H)\subseteq \pi (G)$. It follows from the previous arguments that the group $H^{(k)}$ is also nilpotent, and $\pi(H)=\pi(H^{(k)})$. Applying Lemma~\ref{DirSum} we obtain that
$$
G \le G^{(k)} \le \big( \prod_H H\big)^{(k)}=\prod_H H^{(k)},
$$
which implies the equality $\pi(G)=\pi(G^{(k)})$.\medskip

Now we consider the Sylow subgroups of $G$ and $G^{(k)}$.

\begin{lemma}\label{SylowSubgrpOrbits}
If $P\in\operatorname{Syl}_p(G)$ and $Q\in\operatorname{Syl}_p(G^{(k)})$, then $P^{(k)}\leq Q$ and $\operatorname{Orb}(P)=\operatorname{Orb}(Q)$.
\end{lemma}
\begin{proof}
The inclusion $P^{(k)}\leq Q$ follows from~\cite[Exer.~5.28]{WieInvRel} and the inclusions~(\ref{ClosureSeries}). Therefore, to prove the equality $\operatorname{Orb}(P)=\operatorname{Orb}(Q)$ it suffices to prove that every $P$-orbit is a $Q$-orbit.

Let $\Delta\in\operatorname{Orb}(P)$ and $\Gamma\in\operatorname{Orb}(G)$ such that $\Delta\subseteq\Gamma$. Since the $k$-closure preserves $1$-orbits, $\Gamma$ is also an $G^{(k)}$-orbit. Denote by $\Delta'$ the orbit of the group $Q$ such that
$$
\Delta\subseteq \Delta'\subseteq\Gamma.
$$
The groups $G^\Gamma$ and $(G^{(k)})^{\Gamma}$ are transitive and nilpotent, so the double application of Lemma~\ref{CP1} implies $|\Delta|=|\Gamma|_p=|\Delta'|$, and $\Delta=\Delta'$.
\end{proof}

\begin{lemma}\label{SylowSubgrpClosure}
In above notations, $P^{(k)}=Q$.
\end{lemma}

\begin{proof} It was established in Lemma~\ref{SylowSubgrpOrbits} that $P^{(k)}\leq Q$. To prove the inverse inclusion, let $(\alpha_1,\ldots,\alpha_k)\in\Omega^k$ and $g\in Q$. By Theorem~\ref{W}, there exists $h\in G$ such that $$(\alpha_1,\ldots,\alpha_k)^g=(\alpha_1,\ldots,\alpha_k)^h.$$
Denote by $\Delta_i$ an orbit of $Q$ containing $\alpha_i$, $i=1,\ldots,k$. By Lemma~\ref{SylowSubgrpOrbits}, it follows that every such $\Delta_i$ is also an orbit of~$P$. The element $h$ fixes every $\Delta_i$ as a set, so $\alpha_i^h=\alpha_i^g\in\Delta_i$. In other words, $h\in\bigcap_{i=1}^k G_{\{\Delta_i\}}$, and by Lemma~\ref{SetwiseStabilizers} there exists $u\in P$ such that $h^\Delta=u^\Delta$ (here $\Delta=\bigcup_{i=1}^k \Delta_i$), and then $$(\alpha_1,\ldots,\alpha_k)^g=(\alpha_1,\ldots,\alpha_k)^h=(\alpha_1,\ldots,\alpha_k)^u.$$ It follows from Theorem~\ref{W} that $g\in P^{(k)}$ and $Q\le P^{(k)}$.
\end{proof}

The proof of the theorem is completed by the application of the Lemma~\ref{SylowSubgrpClosure} to the following equalities:
$$
G^{(k)} = \prod_{\operatorname{Syl}(G^{(k)})} Q = \prod_{\operatorname{Syl}(G)} P^{(k)}.\medskip
$$

Churikov Dmitry Vladimirovich

Sobolev Institute of Mathematics SB RAS, 4 Acad. Koptyug avenue,

Novosibirsk State University, 1, Pirogova str.

630090 Novosibirsk Russia

churikovdv@gmail.com

The work is supported by Mathematical Center in Akademgorodok under agreement No. 075-15-2019-1613 with the Ministry of Science and Higher Education of the Russian Federation.
\end{document}